\documentclass[a4paper,12pt]{elsarticle}

\usepackage[T1]{fontenc}
\usepackage{microtype}

\usepackage{cmlgc}
\usepackage{ucs}

\usepackage{color}

\usepackage{amssymb,amsfonts,amsmath,amsthm}
\usepackage[colorlinks=true]{hyperref}
\usepackage{mathrsfs}
\usepackage{graphicx}
\usepackage{rotating}
\usepackage{array}

\usepackage{float}
\usepackage{xspace}

\usepackage[english]{babel}

\newcommand{\Bc}{\mathbf{c}}

\newcommand{\Bv}{\mathbf{v}}

\newcommand{\Bxi}{{\boldsymbol \xi}}
\newcommand{\Btheta}{{\boldsymbol \theta}}
\newcommand{\Balpha}{{\boldsymbol \alpha}}

\newcommand{\cA}{\mathcal{A}}

\newcommand{\cC}{\mathcal{C}}

\newcommand{\cF}{\mathcal{F}}
\newcommand{\cH}{\mathcal{H}}

\newcommand{\cL}{\mathcal{L}}

\newcommand{\cQ}{\mathcal{Q}}
\newcommand{\cR}{\mathcal{R}}

\newcommand{\cT}{\mathcal{T}}
\newcommand{\cU}{\mathcal{U}}
\newcommand{\cV}{\mathcal{V}}
\newcommand{\cW}{\mathcal{W}}

\newcommand{\ue}{\textup{e}}
\newcommand{\uf}{\textup{f}}
\newcommand{\uh}{\textup{h}}
\newcommand{\ux}{\textup{x}}

\newcommand{\ke}{k_{\textup{e}}}
\newcommand{\kf}{k_{\textup{f}}}
\newcommand{\kh}{k_{\textup{h}}}
\newcommand{\kx}{k_{\textup{x}}}

\newcommand{\act}{{\mathbb{A}}} 
\newcommand{\CC}{{\mathbb{C}}}  
\newcommand{\NN}{{\mathbb{N}}}  
\newcommand{\RR}{{\mathbb{R}}}  
\newcommand{\TT}{{\mathbb{T}}}  
\newcommand{\ZZ}{{\mathbb{Z}}}  

\newcommand{\Lie}{\mathcal{L}}

\renewcommand{\act}{Action-Angle\xspace}
\newcommand{\neigh}{neighborhood\xspace}

\newcommand{\SVN}{San V\~u Ng\d{o}c\xspace}

\newcommand{\smooth}{{\cC^\infty}}   

\newcommand{\hypref}[2]{{\hyperref[#1]{#2~\ref{#1}}}}

\newcommand{\ifwork}[1]{\ifthenelse{\boolean{workmode}}{#1}{}}
\newcommand{\comment}[1]{}
\newcommand{\mute}[1]{}
\newcommand{\printname}[1]{}

    \newtheorem{theorem}{Theorem}[section]

\newtheorem{proposition}[theorem]{Proposition}

\newtheorem{corollary}[theorem]{Corollary}
\newtheorem{lemma}[theorem]{Lemma}

\theoremstyle{plain}
\newtheorem{definition}[theorem]{Definition}

\newtheorem{remark}[theorem]{Remark}

\newtheorem{assumption}[theorem]{Assumption}

\newcommand{\ddt}{\frac{d}{dt}}

\newcommand{\re}{\mathop{\mathfrak{R}e}\nolimits}
\newcommand{\im}{\mathop{\mathfrak{I}m}\nolimits}

\newcommand{\flechebas}[1]{%
  \settoheight{\unitlength}{\mbox{$#1$}}%
  \settowidth{\Taille}{\mbox{~${\scriptstyle #1}$}}%
  \addtolength{\unitlength}{4ex}%
  \begin{picture}(0,1)
    \put(0,1){\vector(0,-1){1}}
    \put(0,0.5){\makebox(0,0){${\scriptstyle #1}$ \hspace{\the\Taille}}}
  \end{picture}}
\newcommand{\flechehaut}[1]{%
  \settoheight{\unitlength}{\mbox{$#1$}}%
  \settowidth{\Taille}{\mbox{~${\scriptstyle #1}$}}%
  \addtolength{\unitlength}{4ex}%
  \begin{picture}(0,1)
    \put(0,0){\vector(0,1){1}}
    \put(0,0.5){\makebox(0,0){\hspace{\the\Taille}${\scriptstyle #1}$ }}
  \end{picture}}
\newcommand{\flechedroite}[1]{%
  \settowidth{\unitlength}{\mbox{$#1$}}
  \settoheight{\Taille}{\mbox{${\scriptstyle #1}$}}
  \addtolength{\Taille}{1ex}
  \addtolength{\unitlength}{4ex}
  \raisebox{0.5ex}{%
  \begin{picture}(1,0)
    \put(0,0){\vector(1,0){1}}
    \put(0.5,0){\makebox(0,0){${\scriptstyle #1}$ \vspace{\the\Taille}}}
  \end{picture}}}
\newcommand{\flechegauche}[1]{%
  \settowidth{\unitlength}{\mbox{$#1$}}
  \settoheight{\Taille}{\mbox{${\scriptstyle #1}$}}
  \addtolength{\Taille}{1ex}
  \addtolength{\unitlength}{4ex}
  \begin{picture}(1,0)
    \put(1,0){\vector(-1,0){1}}
    \put(0.5,0){\makebox(0,0){${\scriptstyle #1}$ \vspace{\the\Taille}}}
  \end{picture}}

\newcommand{\cdv}{Colin de Verdi\`ere}

\newcommand{\ouf}{\vspace{8mm}}

\title{Asymptotics of action variables near semi-toric singularities}
\author{Christophe Wacheux\footnote{ Collaborateur scientifique at EPFL, christophe.wacheux@epfl.ch}}

\begin{document}

\begin{abstract}
The presence of focus-focus singularities in semi-toric integrables Hamiltonian systems is one of the reasons why there cannot exist global \act coordinates on such systems. At focus-focus critical points, the Liouville-Arnold-Mineur theorem does not apply. In particular, the affine structure of the image of the moment map around has non-trivial monodromy. In this article, we establish that the singular behaviour and the multi-vauedness of the Action integrals is given by a complex logarithm. This extends a previous result by \SVN to any dimension. We also calculate the monodromy matrix for these systems. 
\end{abstract}


\maketitle

\section{Introduction, definitions and notations} \label{section:Introduction_and_def}

Given a symplectic manifold $(M^{2n},\omega)$, an integrable Hamiltonian system (or IHS) can be defined as a function $F = (f_1,\ldots,f_n):M \to \RR^n$ such that its components are Poisson-commuting and whose differential is of maximal rank almost everywhere. From now, $F$ will always designate for us an IHS. A point $p \in M$ such that $dF(p)$ is of rank $n$ is called a regular point for $F$; it is called critical if otherwise, and in particular it is called fixed if $dF(p) = 0$. We shall note the rank $\kx (p)$, or just $\kx$ if the context is obvious.

For IHS, the famous Liouville-Arnold-Mineur theorem provides a particularly appropriate set of local coordinates near regular points, the \act coordinates. It can be formulated by considering the foliation $\cF$ given by the connected components of the fibers of $F$. The theorem states that for regular leaves of $\cF$ (i.e. leaves without critical points), the germ of foliation is locally a fibration by Lagrangian tori.

The problem is that a generic IHS does have critical points on which one cannot apply Liouville-Arnold-Mineur theorem. One question then, is to examine what can be preserved of the inital result for critical points. Another question is to find the largest open subset of the set of regular points of $F$ on which the period bundle can be trivialized, that is, what are the obstructions to having \emph{global} \act coordinates.

Over the last decades, a lot of work has been produced for both questions. The study of non-degenerate critical points of IHS goes back from the works of Birkhoff and Williamson~\cite{Williamson-OnAlgPbLinNF-1936}, to the works of Bolsinov and Fomenko~\cite{BolsinovFomenko-book}, R\"{u}ss, \cdv, Vey~\cite{ColinVey-LemmeMorseIsochore-1979}, Eliasson, Zung, Miranda, \SVN~\cite{SVN-semiglobalinvariants-2003}, Chaperon~\cite{Chaperon-SmoothFFaSimpleProof-2012}\cite{Chaperon-GeoDiff-SingSysDyn-Asterisque-1986} and many other, including the author~\cite{SVNWacheux-SmoothNF_for_IHS_near_ff_sing-2013}. 
Concerning the existence of global \act coordinates, we cite the works of Duistermaat who established among other obstructions to global \act coordinates the monodromy phenomenon which occurs in particular in our case, and explicited the matrix associated to it~\cite{Duistermaat-globalactionangle-1980},\cite{DuistermaatHeckman-variationincohomology-1982}. Dazord and Delzant~\cite{DazordDelzant-Leproblemegeneral-1987} extended the study to coisotropic foliations, while Dufour, Molino and Toulet began to study the case with critical points \cite{DufourMolino-CompactActionRnVarAA-1991},\cite{DufourMolinoToulet-ClassIHSetInvFomenko-1994}. 

This article deals mostly with the first question, yet the two questions are deeply linked. As we shall see, once we have the proper local model to describe what occurs near focus-focus singularities, the monodromy matrix becomes very easy to calculate. We consider mainly singularity of maximal corank (see precise definition in section~\ref{subsection : toric and semi-toric system} for precise definition). Our main result is the existence of suitable coordinates, in which the action integrals near the focus-focus singulaity have a simple expression where the singular behaviour and the multi-valuedness is expressed by a simple complex logarithm (Theorem~\ref{theo:AA_coord_around_FF_value}). With these coordinates, it is possible to compute the monodromy matrix. The two results we establish here extend the result proved in \cite{SVN-semiglobalinvariants-2003} for the case $2n=4$ to higher dimension and with possible elliptic components. 

\ouf

The article is organised as following: first, we set the necessary notions nedded for a precise formulation of the result. Next, we present in details the counterpart of Morse theory for integrable Hamiltonian systems at the local and semi-global scale. Then, we prove our main result, compute the topological monodromy matrix and give a comment to the case with elliptic components.

\section{Statement of the result}

In order to give a precise formulation of our result, we need to recall the notion of non-degeneracy for integrable Hamiltonian systems. We shall discuss the counterpart of the Morse theory that we obtain in that framework. 

\subsection{Critical points in integrable Hamiltonian systems}

Remember first that $\smooth((M,\omega) \to \RR)$ is naturally equipped with a Poisson bracket $\{.,.\}$ such that $(\smooth((M,\omega) \to \RR),\{.,.\})$ is a Lie algebra. At a fixed point $p$ of a function $u \in \smooth(M \to \RR)$, one can associate a quadratic form $\cH[u]_p \in S^2(T_p M)^*$ by taking the Hessian of $u$ in a local set of coordinates. It is well defined (it does not depends of the local coordinates) because $p$ is a fixed point. 

The symplectic form $\omega$ induces a Poisson bracket $\{.,.\}_p$ on $S^2(T_p M)^*$ in the following way: 

\[ \{ \cH[f]_p,\cH[g]_p \}_p := \cH[\{f,g\}]_p,  \]
and $(S^2(T_p M)^*,\{.,.\}_p)$ is a Lie algebra isomorphic to $sp(2n,\RR)$, the Lie algebra of the symplectic group.
Now, reminding that a Cartan subalgebra of a Lie algebra $A$ is a subalgebra of $A$ which is abelian and self-centralizing, we can define the following

\begin{definition}
 A fixed point $p$ for $F$ is said to be \emph{non-degenerate} if $\langle \cH[f_1]_p,\ldots,\cH[f_n]_p \rangle$, the subalgebra spanned by the Hessians of the components of $F$ at $p$, is a Cartan subalgebra of $(S^2(T_p M)^* , \{.,.\}_p)$. 
\end{definition}

Now, to define a non-degeneracy condition for critical points of arbitrary rank, we remark that $S^2(T_p M)^*$ is isomorphic as a Lie algebra to a subalgebra of $\cQ(\RR^{2n} \to \RR)$, the algebra of quadratic forms of $\RR^{2n}$. We consider a critical point $m$ of $M$ for $F$ of rank $\kx$, and we may assume without loss of generality that $df_{n-\kx+1} \wedge \ldots \wedge df_n \neq 0$. We can thus apply Darboux-Caratheodory theorem to the system $\langle f_{n-\kx+1},\ldots,f_n \rangle$: there exists a symplectomorphism $\varphi: (\cU,\omega) \to (\RR^{2n}, \sum_{i=1}^n d\tilde{f_i} \wedge du_i)$ with $\varphi(m) = 0$ and such that $\tilde{f_j} = f_j \circ \varphi^{-1} - f_j(m)$ are canonical coordinates $\xi_j$ for $j \geqslant n-\kx+1$. In these local coordinates, since the $f_j$ are Poisson commuting, $f_1,\ldots,f_{n-\kx}$ do not depend of $x_{n-\kx+1},\ldots,x_n$. We define the function $g_j : \RR^{2(n-\kx)} \to \RR$, $1 \leqslant j \leqslant n - \kx$ as $g_j (\bar{x},\bar{\xi}) = f_j'(\bar{x},0,\bar{\xi},0)$. 

\begin{definition}
 A critical point of rank $\kx$ is called \emph{non-degenerate} if for the $g_i$'s defined above, the Hessians $\cH [g_i]_0$ span a Cartan subalgebra of $(\cQ(\RR^{2(n-\kx)} \to \RR),\varphi_* (\{.,.\}_p))$. A Hamiltonian system is called non-degenerate if all its critical points are non-degenerate. 
\end{definition}


\subsection{Toric and semi-toric systems} \label{subsection : toric and semi-toric system}

Relying on the decomposition of Cartan subalgebras of $sp(2n)$, one can give the following classification result of non-degenerate critical points due to Williamson~\cite{Williamson-OnAlgPbLinNF-1936}:

\begin{theorem}
 
 Let $p \in M$ a nondegenerate critical point for $F$ an IHS. Then there exists a symplectomorphism $\varphi: (M,\omega,p) \to (\RR^{2n},\omega_0 = \sum_{i=1}^n d\xi_i \wedge dx_i,0)$ such that
 \[ \varphi^*F = (\ue_1,\ldots,\ue_{\ke},\uh_1,\ldots,\uh_{\kh},\uf^1_1,\uf^2_1,\ldots,\uf^1_{\kf},\uf^2_{\kf},\xi_{n-\kx+1},\ldots,\xi_n) + o(2) \text{ with:}\]
 
 \begin{itemize}
  \item $\ue_i = x^2_i + \xi^2_i$
  \item $\uh_i = x_i y_i$
  \item $\begin{cases}
         \uf^1_i = x^1_i \xi^1_i + x^2_i \xi^2_i \\
         \uf^2_i = x^1_i \xi^2_i - x^2_i \xi^1_i
        \end{cases}$
 \end{itemize}
 
\end{theorem}

The theorem introduces three classes of possible ``components'' at a critical point (apart from the regular components $\xi_i$) : elliptic (the $\ue_i$'s), hyperbolic (the $\uh_i$'s), and focus-focus (the couples $\uf_i = (\uf^1_i,\uf^2_i)$). We can thus define the following notations

\begin{definition}

Given $(M,\omega,F)$ an IHS, and we can associate to $m \in M$ the following \emph{Williamson type} (with respect to $F$), or \emph{Williamson index} $\Bbbk = (\ke,\kf,\kh,\kx) \in \NN^4$ with

 \begin{itemize}
 \item $\ke -$ number of \emph{elliptic} (or $E$) components, 
 \item $\kf -$ number of \emph{focus-focus} (or $FF$) components,
 \item $\kh -$ number of \emph{hyperbolic} (or $H$) components,
 \item $\kx -$ number of \emph{transverse} (or $X$) components, that is the regular components.
 
\end{itemize}

We may also use the notation $FF^{\kf}-E^{\ke}-H^{\kh}-X^{\kx}$ instead of $\Bbbk$. We also define 

\[Q_\Bbbk := (\ue_1,\ldots,\ue_{\ke},\uf^1_1,\uf^2_1,\ldots,\uf^1_{\kf},\uf^2_{\kf},\uh_1,\ldots,\uh_{\kh},\xi_1,\ldots,\xi_{\kx})\] 

\end{definition}

Note that the four coefficients are linked by the following equation

\begin{equation} \label{equ:Williamson_type}
 \ke + 2\kf + \kh + \kx = n
\end{equation}

\begin{definition}
The set $\cW(F)$ is defined as the set of different Williamson types that occurs for a given IHS $F$. When equipped with the following relation \[ \Bbbk \preccurlyeq \Bbbk' \text{ if : } \ke \geqslant \ke', \kf \geqslant \kf' \text{ and } \kh \geqslant \kh' ,\] it is a (partially) ordered set (the term \emph{poset} also appears in the litterature). 

\end{definition}

Let us show that $(\cW(F),\preccurlyeq)$ is an ordered set. Let $\Bbbk, \Bbbk', \Bbbk'' \in \cW(F)$.

\begin{itemize}
 \item {\bf reflexivity:} we always have $\ke \geqslant \ke, \kf \geqslant \kf$, and $\kh \geqslant \kh$, thus $\Bbbk \preccurlyeq \Bbbk $,
 \item {\bf antisymetry:} if $\Bbbk \preccurlyeq \Bbbk'$ and $\Bbbk' \preccurlyeq \Bbbk$, then $ \ke \geqslant \ke'$, $\kf \geqslant \kf'$, $\kh \geqslant \kh'$ and $ \ke \leqslant \ke'$, $\kf \leqslant \kf'$, $\kh \leqslant \kh'$, so $ \ke = \ke'$, $\kf = \kf'$, $\kh = \kh'$ and hence, by equation~\ref{equ:Williamson_type} we have $\kx = \kx'$, so $\Bbbk = \Bbbk'$,
 \item {\bf transitivity:} if $\Bbbk \preccurlyeq \Bbbk'$ and $\Bbbk' \preccurlyeq \Bbbk''$ then $ \ke \geqslant \ke' \geqslant \ke'', \kf \geqslant \kf' \geqslant \kf'', \kh \geqslant \kh' \geqslant \kh''$, hence $\Bbbk \preccurlyeq \Bbbk''$.
 \end{itemize}

This is the ordered set involved in the stratification mentioned in section~\ref{subsection:the_general_case}. We can also define consistently the Williamson type of a leaf (see~\ref{subsection:Semi-global_theory}) as follows

\begin{definition} \label{def:Williamson_type_leaf}
Given a leaf $\Lambda \in \cF$, the Williamson type $\Bbbk(\Lambda)$, or $\Bbbk$ if it is unambiguous, is the Williamson type of the point of smallest rank.
\end{definition}

Lastly, we introduce the useful following sets

\begin{definition}
Let $\cU \in M$ be an open set. We define $P_\Bbbk (\cU) \subseteq \cU$ (resp. $L_\Bbbk(\cU) \subseteq M$) is the set of critical points (resp. leaves) of type $\Bbbk$ in $\cU$. Finally, we define $V_\Bbbk(\cU) := F(L_\Bbbk(\cU))$, the set of critical values in $F(\cU)$ of Williamson type $\Bbbk$.

\end{definition}

\begin{remark}
In the above definition, $P_\Bbbk$ (resp. $L_\Bbbk$, resp. $V_\Bbbk$) are covariant functors from the category of open sets with inclusions as morphisms, to the categories of subsets of $M$( resp. union of leaves of $\cF$, resp. subsets of $F(M)$), with inclusions as morphisms.
Since the fibers of the moment map are not a priori connected, a critical value may belong to $V_\Bbbk (\cU)$ and $V_{\Bbbk'} (\cV)$ with $\Bbbk$ and $\Bbbk'$ possibly different, and $\cU$ and $\cV$ possibly disjoint.

\end{remark}

The Williamson type is a symplectic invariant. The aim of this article is to examine the asymptotic behavior of \act coordinates when getting close to a singularity with one focus-focus component and no hyperbolic component. 

This question is actually motivated by a long-term program of classification of IHS based on their dynamical behaviour. While such classification for general systems is out of reach for now, partial results exists for subclasses of IHS. To formulate some of these results, we introduce here a criterium called ``complexity''. The notion of complexity find its origins in the works of Karshon an Tolman~\cite{KarshonTolman-Centeredcomplexity1HamTorusAction-2001}\cite{KarshonTolman-CompleteinvariantsforHamT_actions_tall-2003}\cite{KarshonTolman-ClassificationofHamiltonian-2011}, Margaret Symington and Leung~\cite{Symington-4from2-2001}, \cite{LeungSymington-Almosttoricsymplectic-2010}, and of \SVN in~\cite{SVN-Momentpolytopessymplectic-2007}.

\begin{definition}
 Let $F=(f_1,\ldots,f_n):(M^{2n},\omega) \to \RR^n$ be an integrable Hamiltonian system. It is said to be \emph{almost-toric} of complexity $c \leqslant n$ if (up to a global permutation of the components of $F$), every critical point verifies these conditions:
 \begin{itemize}
  \item all critical points are non-degenerate.
      
  \item there are no singularities of hyperbolic type: $\kh = 0$,
  \item the function $\check{F}^c=(f_{c+1},\ldots,f_n)$ generates a \emph{global} $\TT^{n-c}$-action.
  \end{itemize}

 An almost-toric system is called \emph{semi-toric} if $c=1$, and \emph{toric} if $c = 0$. For a semi-toric system, $\check{F}:=(f_2,\ldots,f_n)$ is the map that generates the $\TT^{n-1}$-action.
 
\end{definition} 
 
 The classification of toric IHS begins with Liouville Arnold Mineur theorem : the \act coordinates provide an (integral) affine structure on the base space. It was shown in 1982, simultaneously by Atiyah~\cite{Atiyah-ConvexityandCommuting-1982}, and Guillemin \& Sternberg~\cite{Atiyah-ConvexityandCommuting-1982}, \cite{GuilleminSternberg-ConvexitypropertiesI-1982}, \cite{GuilleminSternberg-ConvexitypropertiesII-1984}, that for an Hamiltonian $\TT^d$-action, the image of the moment map is a rational convex polytope in $\RR^d$. Delzant, in 1988, gave a constructive proof that, in the completely integrable case $d=n$, if the action is effective, the polytope completely determines the system, thus completing the classification for toric IHS \cite{Delzant-Hamiltoniensperiodiqueset-1988} \cite{Delzant-ClassificationActions-1990}.

 
 For the semi-toric case though, the image of moment map is not a polytope, and it is not enough to classify such IHS. However, \SVN and Pelayo use it to get a classification ``\`{a} la Delzant'' for semi-toric systems in dimension $2n=4$~\cite{SVN-semiglobalinvariants-2003},\cite{PelayoSVN-Semitoricintegrablesystems-2009},\cite{PelayoSVN-ConstructingIntSysOfSemitoricType-2011}. Among the other classifying invariants they introduce, there is a formal series that is explicited as the Taylor expansion of regularized Action coordinate, thus describing, in a way, how the Lagrangian fibration pinches at a semi-toric singularity. The principal goal of this article is to give the general formula of this invariant, for any semi-toric critical point in any dimension.
 
 \subsection{The main result}
 
 A semi-toric IHS only has elliptic critical points, which are well understood from the study of toric system, and critical points of Williamson type $FF-X^{\kx}-E^{\ke}$. From~\cite{Wacheux-LocalmodelsofS-TIntSys-2014} we know the set $\Gamma := V_{FF-X^{\kx}-E^{\ke}}$ is a $\kx$-dimensional submanifold. We note $D^r In t$ the disk of dimension $r$. In this article, we prove the following result. 
 
\begin{theorem} \label{theo:AA_coord_around_FF_value}
 Let $(M^{2n},\omega,F=(f_1,\ldots,f_n))$ be a semi-toric integrable system with $n$ degrees of freedom. Let $m$ be a $FF-X^{n-2}$ critical point, $\Lambda_0$ the leaf containing $m$ and $\cU(\Lambda_0)$ (a germ of) a tubular \neigh of $\Lambda_0$ such that $V_{X^n}(\cU(\Lambda_0))$ is simply connected.
 
 There exists (a germ of) a tubular \neigh $W \subseteq \RR^n$ of $W_0 := \{0\} \times \{0\} \times D^{n-2}$, a local diffeomorphism
 \[G: (W, W_0) \to (F(\cU(\Lambda_0)),\Gamma) ,\] and a symplectomorphism 
 \[ \varphi: \left( L_{X^n} (\cU(\Lambda_0)) , \omega \right) \to \left( (W \setminus W_0) \times \TT^n, \tilde{\omega} \right) \]
 
 such that if we write $\varphi = (I_1,\ldots,I_n,\theta_1,\ldots,\theta_n)$ we have
 
 \begin{enumerate}
  \item The coordinates $(\theta_3,\ldots,\theta_n,I_3,\ldots,I_n)$ can be extended to partial \act coordinates on $\cU(\Lambda_0)$.
  \item We have that
  \[
  \begin{aligned}
  I_1(\Bv) & = S(\Bv) - \re (w \ln(w) - w) \\
  I_2(\Bv) & = v_2 \\
  I_3 & = v_3 \\
  \ & \ \vdots \\
  I_n & = v_n
  \end{aligned}
  \]
  where $\Bv = (v_1,\ldots,v_n) = G^{-1} \circ F$, $w = v_1 + i v_2$, $\ln$ is a determination of the complex logarithm on $W_{v_1,v_2} = W \cap \{ v_3 = \ldots = v_n = 0 \} \subseteq \RR^2 \simeq \CC$, and $S$ a smooth function of $\Bv$.
 \end{enumerate}
 
\end{theorem} 
 
 From this therorem in the case of maximal codimension, we treat the general case with possible elliptic components. While it is quite intuitive, we shall define precisely the meaning of ``partial \act coordinates'' in section~\ref{subsection:Semi-global_theory}. To prove this result, some genericity assumptions are necessary, which we shall indicate during the proof.
 
 This result was obtained during my Ph.D. thesis. We present it here in the simplest manner, with the addition of the case wwith elliptic components
 
\section{A symplectic Morse theory for integrable Hamiltonian systems}

The non-degeneracy condition we defined calls for an equivalent of the Morse theory in the symplectic framework and integrable Hamiltonian systems. We shall see that there are different level at which we can establish counterparts of classical results in Morse theory. 

\subsection{ Local and orbital theory} \label{subsection:local_and_orbital_theory}

The first difference with classical Morse theory is that instead of a single function, we have a family of $n$ real-functions. We exploit the relations between the functions to get a normal form theorem, due to Eliasson. The version of the theorem given here is an extension of the original theorem which incorporates partial \act coordinates for the transverse components. This version is due to Miranda and Zung~\cite{MirandaZung-Equivariantnormalform-2004}. Also, we give the version only for semi-toric singularites 

\begin{theorem}[Eliasson-Miranda-Zung Normal Form - Semi-toric case] \label{theo:Eliasson_NF-ST_case}
 Let $(M^{2n},\omega,F)$ a semi-toric integrable system with $n \geqslant 2$, $F$ an IHS. Let $m$ be a critical point of Williamson type $\Bbbk = FF - E^{\ke} - X^{\kx}$ and $F_X := (f_{n-\kx+1},\ldots,f_n)$ the transverse components in Williamson decomposition of $F$ at $m$ that provide a global $\TT^{n-2}$-action on $M$.  
 
 Then there exists a triplet $(\cU_m,\varphi_\Bbbk,G_\Bbbk)$ with $\cU_m$ an open \neigh of $m$ saturated with respect to the orbit $\cT$ of $F_X$, $\varphi_\Bbbk$ a symplectomorphism of $\cU_m$ to a \neigh of $(0 \in \RR^{2n} , \omega_0)$ that sends the orbit $\cT \simeq \TT^{\kx}$ of $m$ to $O$, and $G_{\Bbbk}$ a local diffeomorphism of $0 \in \RR^n$ such that:
 
 \[ \varphi^*_{\Bbbk} F = G_{\Bbbk} \circ Q_{\Bbbk} \]
  
\end{theorem}
 
This is a counterpart of the Morse lemma, here adapted to the symplectic context: in a \neigh of a non-degenerate critical point, an integrable system is, up to a regular change of coordinates, a function of of the quadratic parts determined by the Williamson type of the critical point. As we said before, we have here $\kx$ \act coordinates, thus allowing us to get a normal form on a ``wider'' open set. This is what one may call a semi-local, or an orbital result.

It was the contribution of many people that allowed eventually the statement and proof of the original theorem of Eliasson for fixed points. The first works to be cited here are those of Birkhoff, Vey~\cite{Vey-SurCertainsSystemes-1978}, Colin de Verdi\`{e}re and Vey~\cite{ColinVey-LemmeMorseIsochore-1979}, and of course Eliasson in~\cite{Eliasson-Thesis-1984} and~\cite{Eliasson-NormalformsHamiltonian-1990}. More recently, Chaperon in~\cite{Chaperon-SmoothFFaSimpleProof-2012} and~\cite{Chaperon-GeoDiff-SingSysDyn-Asterisque-1986}, Zung in \cite{Zung-Anoteonfocusfocus-1997} and~\cite{Zung-Anothernotefocus-2002}, and \SVN \& the author in~\cite{SVNWacheux-SmoothNF_for_IHS_near_ff_sing-2013} provided new proofs and filled the technical gaps that remained in the original proof. Miranda and Zung in \cite{MirandaZung-Equivariantnormalform-2004}, relying on Eliasson local normal form provided the orbital version. 

Eliasson normal form allows us to visualize the image of a neighborhoud of non-degenerate critical points. Here is a picture of the different sets of critical values that can occur in dimension $2n=4$.

\begin{figure}[ht]
\centering
\includegraphics[width=\linewidth]{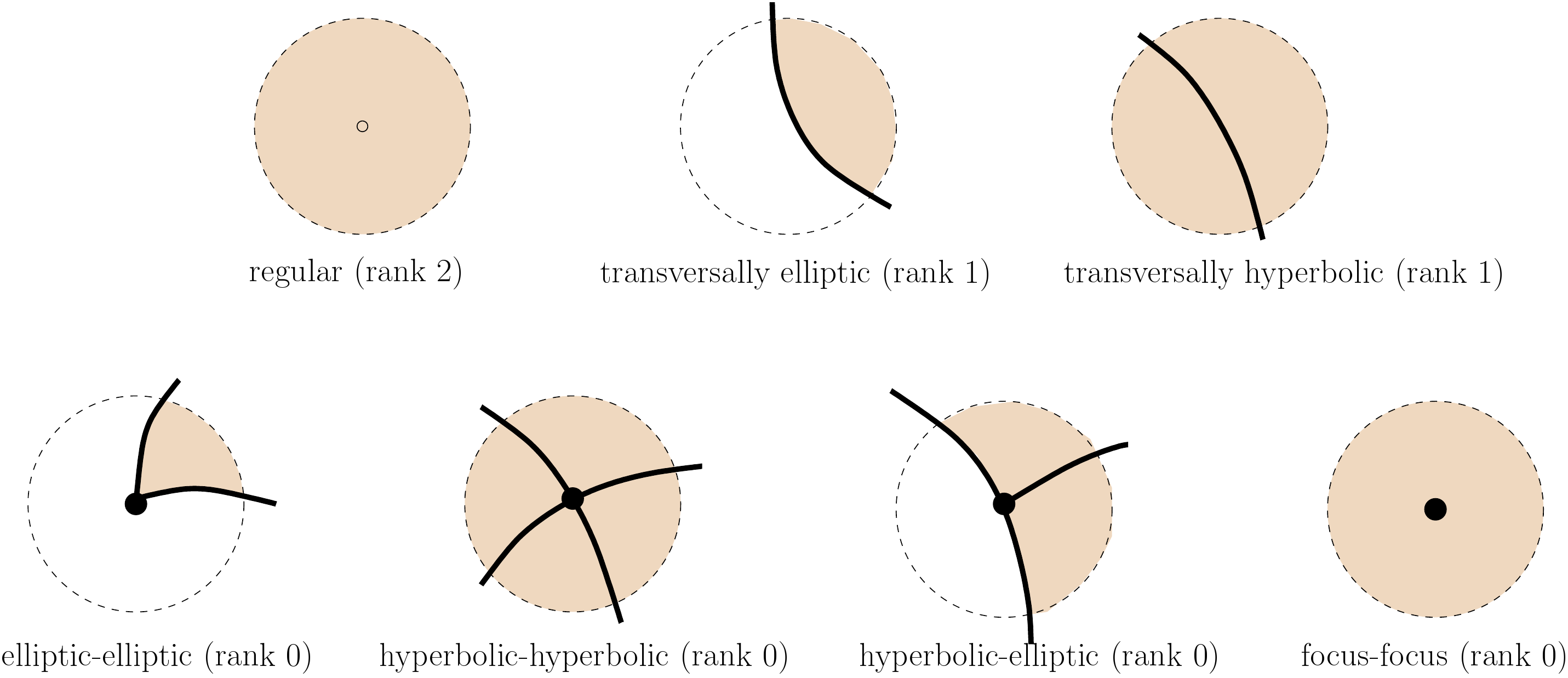}
 \caption{Local models of critical points in dimension 2, \cite{PelayoRatiuSVN-SymplecticBifTheoryForIntegrableSystems-2011}}
\end{figure}



\subsubsection{ Dynamic of the local models } \label{subsection:linear_dynamic}

A description of the Lagrangian torus foliation near a semi-toric leaf starts with the study of the dynamic near the critical point. For each type of critical component, it is always possible to introduce a local model using complex coordinates. For elliptic and hyperbolic components, we introduce natural complex coordinates established after Darboux coordinates $z = x + i\xi$, along with their respective elliptic $\ue = z\bar{z})$ and hyperbolic $\uh = \im(-\frac{1}{2} z^2)$ Hamiltonian. However, for focus-focus components we take the following 4-dimensional model: 

\[ \begin{aligned}
(\RR^4,0) & \xrightarrow{ \; \cong \;} (\CC^2,0) \\
(x_1,x_2,\xi_1,\xi2) & \longmapsto ( \underbrace{x_1 + i x_2}_{=: z_1} , \underbrace{\xi_1 + i\xi_2}_{=: z_2})  \\
\omega = dx_1 \wedge d\xi_1 + dx_2 \wedge d\xi_2 & \longmapsto \omega = \re(d\bar{z_1} \wedge dz_2)\\
\uf = (\uf^1,\uf^2) & \longmapsto \uf^1 + i \uf^2 = \bar{z_1}z_2 . 
\end{aligned}
\]

We summarize in the array below some of the dynamical properties of each component. 

\begin{table}[ht] 
\begin{tabular}{|p{2.2cm}|p{4cm}|p{5.9cm}|}
 \hline
 Component & Critical leaf & Expression of the flow in local coordinates \\
 \hline 
 Elliptic & $\{ \ue = 0 \}$ is a point in $\RR^2$ & $\phi_{\ue}^t : z^\ue \mapsto e^{it} z^\ue$ \\
 \hline
 Focus-focus & $\{ f^1 = f^2 = 0 \}$ is the union of the planes $\{x_1 = x_2 = 0 \}$ and $\{\xi_1 = \xi_2 = 0 \}$ in $\RR^4$ & $\phi_{f^1}^t : (z^f_1,z^f_2) \mapsto (e^{-t} z^f_1, e^t z^f_2)$ and 
 $\phi_{f^2}^t : (z^f_1,z^f_2) \mapsto (e^{it} z^f_1, e^{it} z^f_2)$ \\
 \hline
\end{tabular}
\caption{Properties of elementary blocks}
\label{tab:properties_of_crit_points}

\end{table}

Note also that $\uf$, which is here the complex counterpart to the Hamiltonian, reminds of a hyperbolic singularity. This is why focus-focus critical points are also called ``complex hyperbolic'' or also ``loxodromic'' in the literrature. 

We give below a representation of the dynamic near an elliptic and a focus-focus critical point. For the elliptic case, the vector field $\chi_\ue$ is just the rotation around the critical point. For the focus-focus case, each vector field acts simultaneously on the two complex planes. The ``pseudo-hyperbolic'' field $\chi_{\uf^1}$ is along radial trajectories on each plane, that is, half-lines starting at the focus-focus critical point, while the ``pseudo-elliptic'' field $\chi_{\uf^2}$ is just the rotation around the focus-focus critical point. Here is a picture of these two fields in dimension $2n=4$, with $\chi_1 = \chi_{f^1}$ and $\chi_2 = \chi_{f^2}$.

\begin{figure}[ht]
\centering
 \includegraphics[width=0.6\linewidth]{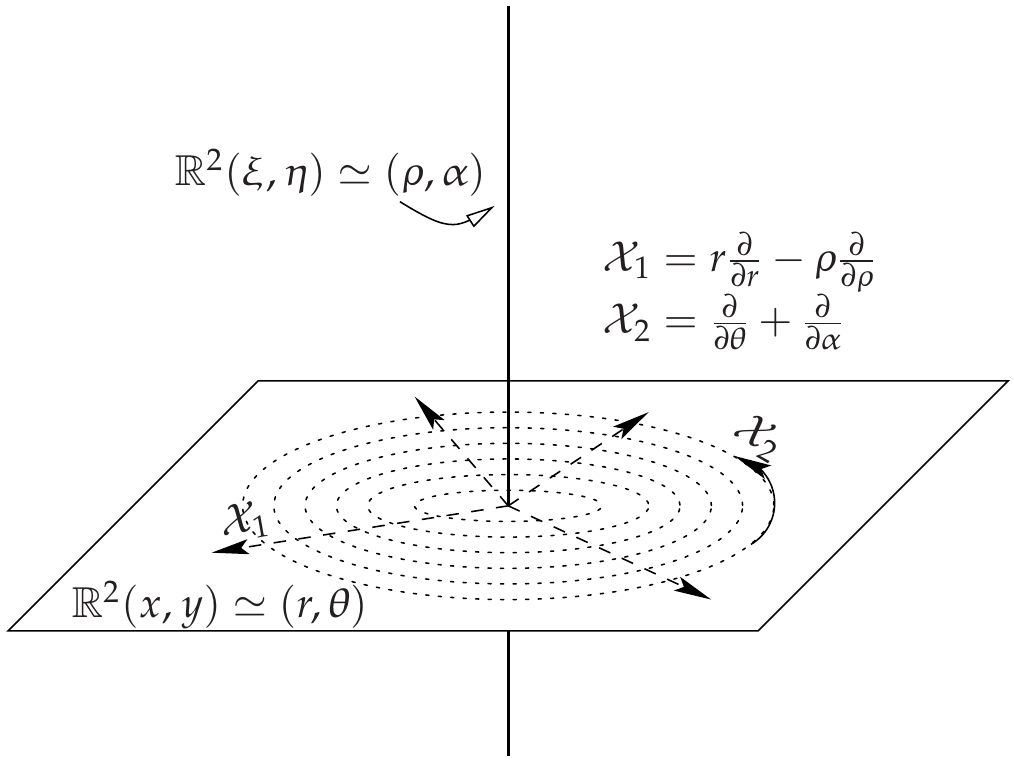}
 \caption{Linear model of focus-focus critical point, \cite{SVN-BohrSommerfeld-2000} }
\end{figure}

\newpage

\subsection{Semi-global (leaf-wise) theory} \label{subsection:Semi-global_theory}

 Our goal here is to formulate a version of the normal form for a critical leaf, and the foliation near it, that is, a Liouville-Arnold-Mineur theorem for critical points. There are several problems we need to examine in order to be able to formulate  The first problem is the following. In symplectic geometry, the orbit of a point $m$ by the $\RR^n$ action induced by the moment map is contained in the leaf that contains $m$, and it is easy to see that each point of an orbit has the same Williamson type. However, while for elliptic critical points the orbit is exactly the leaf, for hyperbolic and focus-focus critical points this is not the case. In the semi-toric case in dimension $2n=4$, for instance, the leaf containing the focus-focus point contains also points for which the action is of rank 2 (in the local model, they are the points for which $z_1$ or $z_2$ is not zero). 
 
 Fortunately, in~\cite{Zung-SymplecticTopology_I-1996}, Zung describes the stratification of a given leaf $\Lambda \in \cF$ by the orbits of its points, and proves that all the orbits with lowest rank have the same Williamson type, and it is thus an invariant of the leaf. This justifies Definition~\ref{def:Williamson_type_leaf}

 
 One important point also when we consider semi-toric leaves, is that a semi-toric leaf (a leaf that contains a critical point with $\kf =1$) may carry more than one semi-toric orbit. We will not consider this case for the present time
 
\begin{assumption}
 From now on, for $F$ a semi-toric IHS, the leaves of $F$ will only carry at most one semi-toric orbit.
\end{assumption}

In order to expose Arnol'd-Liouville theorem for critical points, we need some definitions.
 
\begin{definition}
 A (non-degenerate) \emph{singularity} shall be defined as (a germ of) a tubular \neigh of a (non-degenerate) leaf. We adopt here the following notation: for $c$ in $B$ the base space of $\cF$ and $\pi_\cF:M \to B$ the projection map, $\Lambda_c := \pi_\cF^{-1}(c)$ is the leaf of all points in $M$ over $c$ in the fibration, and $\cU(\Lambda_c)$ a tubular \neigh of $\Lambda_c$.
 
 Two singularities are isomorphic if they are leaf-wise isomorphic.
\end{definition}

There is a mild assumption concerning critical leaves that is required to formulate the central result of the theorem.

\begin{definition}
 A non-degenerate critical leaf $\Lambda$ is called \emph{topologically stable} if there exists a tubular \neigh $\cV$ of $\Lambda$ and a $\cU \subset \cV(\Lambda)$ a small \neigh of a point $m$ of minimal rank, such that 
\[ \forall \Bbbk \in \cW (F),\ V_\Bbbk (\cV(\Lambda))  = V_\Bbbk (\cU)) \ . \]

 An integrable system will be called \emph{topologically stable} if all its critical points are non-degenerate and topologically stable.
\end{definition}

\begin{assumption} \label{assumption:topologically_stable}
 From now on, all of our systems will be topologically stable.
\end{assumption}

The assumption of topological stability rules out some pathological behaviours that can occur for general foliations. Note however that for all known examples, the non-degenerate critical leaves are all topologically stable, and it is conjectured that this is also the case for all analytic systems. 




\begin{definition} \label{def:simple_leaf}

\begin{itemize} 
 
 \item A singularity is called of \emph{(simple) elliptic type} if it is isomorphic to $\cL^e$: a plane $\RR^2$ foliated by $e = x^2 + \xi^2$.
 \item A singularity is called of \emph{(simple) focus-focus type} if it is isomorphic to $\cL^f$, where $\cL^f$ is given by $\RR^4$ locally foliated by $f^1 = x_1 x_2 + \xi_1 \xi_2$ and $f^2 = x_1 \xi_2 - x_2 \xi_1$. 
\end{itemize}
\end{definition}

Topological properties of simple elliptic, hyperbolic and focus-focus singularities are discussed in details in~\cite{Zung-SymplecticTopology_I-1996},\cite{Zung-Anoteonfocusfocus-1997},\cite{Zung-Anothernotefocus-2002},\cite{SVN-BohrSommerfeld-2000}. In particular, one can show that the focus-focus critical leaf must be homeomorphic to a pinched torus that we note $\check{\TT}^2$; it is topologically equivalent to a 2-sphere with two points identified. The regular leaves around are regular tori.


\subsubsection{Singular Arnold-Liouville theorem}

Now we can formulate an extension of Liouville-Arnold-Mineur theorem to singular leaves. We call the next theorem a ``leaf-wise'' result, as the results given hold for a \emph{leaf} of the system. However assertion $2.$ of the theorem does not extend Eliasson normal form, since here the normal form of the leaf doesn't preserve the symplectic structure. Again, we only give the semi-toric version.

\begin{theorem}[Arnold-Liouville with semi-toric singularities,~\cite{Zung-SymplecticTopology_I-1996}] \label{theo:A-L_with_sing}
Let $F$ be a proper semi-toric system, $\Lambda$ be a non-degenerate critical leaf of Williamson type $\Bbbk$ and $\cV(\Lambda)$ a tubular \neigh of $\Lambda$.

Then the following statements are true:
\begin{enumerate}
 \item There exists an effective Hamiltonian action of $\TT^{\ke +\kf +\kx}$ on $\cV(\Lambda)$. There is a locally free $\TT^{\kx}$-subaction. The number $\ke + \kf + \kx$ is the maximal possible for an effective Hamiltonian action.
 \item If $\Lambda$ is topologically stable, on $(\cV(\Lambda)$ the foliation $\cF$ is leaf-wise homeomorphic (and even diffeomorphic) to a product of elliptic and focus-focus simple singularities:
 
 \[ (\cV(\Lambda),\cF) \simeq ( \cU(\TT^{\kx}),\cF_\ux ) \times \cL^\ue_1 \times \ldots \times \cL^\ue_{\ke } \times \cL^\uf_1 \times \ldots \times \cL^\uf_{\kf}  \]
where $(\cU(\TT^{\kx}),\cF_\ux)$ is a foliation of the full torus $( \cU(\TT^{\kx}),\cF_\ux ) \subseteq \RR^{2\kx}$ by tori $\TT^{\kx}$.

 \item There exists partial \act coordinates on $\cV(\Lambda)$ : there exists a diffeomorphism $\varphi$  on $\cV(\Lambda)$ such that
 \[ \varphi^* \omega = \sum_{i=1}^{\kx} d\theta_i \wedge dI_i + P^* \omega_1 \]
 where $(\theta_1,\ldots,\theta_{\kx},I_1,\ldots,I_{\kx})$ are the action-angle coordinates on $T^* \cT$ ($\cT$ being $F_X$-orbit in Eliasson-Miranda-Zung Theorem~\ref{theo:Eliasson_NF-ST_case}), and $\omega_1$ is a symplectic form on $\RR^{2(n-\kx)} \simeq \RR^{2(k_{\ue } + 2\kf)}$.
\end{enumerate}
\end{theorem}
 
 In Definition~\ref{def:simple_leaf}, the description of the simple focus-focus leaf is actually a consequence of the existence of a Hamiltonian $S^1$-action on a tubular \neigh of a simple focus-focus singularity. It is important to note that in statement 2. of Theorem~\ref{theo:A-L_with_sing}, the decomposition is at best leaf-wise diffeomorphic, but not symplectomorphic. 

\section{Proof of the main result}
 
 We solve the case in Theorem~\ref{theo:AA_coord_around_FF_value}, and complete it with a comment on the case where some components are elliptic.
 
 \subsection{The transversally focus-focus case}
 
 Let us consider a $FF-X^{n-2}$ singularity $\cU(\Lambda_0)$ of our foliation $\cF$. We already know with item 3. of Theorem~\ref{theo:A-L_with_sing} that there exists partial \act on $\cU(\Lambda)$ that are well defined. They are \act coordinates associated to the $\TT^{n-2}$-action induced by the transverse components of $F$. For the two other coordinates, we have that

\begin{enumerate}
 \item one cannot define \act coordinates on the  of $\cF$,
 \item one can only define \act coordinates canonically on the set of regular leaves $L_{X^n} (\cU(\Lambda))$, if $V_{X^n}(\cU(\Lambda))$ is simply connected: this is the monodromy phenomenon. 
\end{enumerate}

To show these two points, we define a complete set of \act coordinates on the regular tori foliation near a focus-focus singularity, and give the asymptotics of the Action coordinate near the critical point. To prove the theorem, we follow and generalize each step of the proof of Section~3 in \cite{SVN-semiglobalinvariants-2003}. During the proof will arise what causes the monodromy phenomenon.

\begin{proof}{of Theorem}~\ref{theo:AA_coord_around_FF_value}

With Zung's theorem, we can always take $\cU(\Lambda_0)$ small enough so that for all $\Bv \in F(\cU(\Lambda_0))$, there is a unique leaf in $\cU(\Lambda_0)$, that is, for the restricted system $(M,\omega,\cF|_{\cU(\Lambda_0)})$ the fibers are connected. We then define, for $\Bv \in F(\cU(\Lambda_0))$, $\Lambda_\Bv = F^{-1}(\Bv)$.


Item $2.$ of Theorem~\ref{theo:A-L_with_sing} gives us the topological (and differential) description of the foliation on the singularity $\cU(\Lambda_0)$, while Eliasson-Miranda-Zung normal form (Theorem~\ref{theo:Eliasson_NF-ST_case}) describes it symplectically, but only on a \neigh $\cU_{MZ}$ of $m$, $\cU_{MZ}$ stable by the flow of $\check{F}$. On $\cU_{MZ}$ there is a local symplectomorphism $\varphi: \cU_{MZ}  \to T^* \RR^2 \times T^* \TT^{n-2}$ and a local diffeomorphism $G:\RR^n \to \RR^n$  such that, for $Q_{FF-X^{n-2}} := (f^1,f^2,\xi_1,\ldots,\xi_{n-2})$, 

\[ \varphi^* F  = G \circ Q_{FF-X^{n-2}} \]

From this form, we can see first that locally $\Gamma$ is indeed a $\kx$-dimensional submanifold.


Now let us have a point $A_0 \in \cU_{MZ} \cap \Lambda_0$ different than $m$: $A_0$ is on the same critical fiber as $m$, and near enough so that Eliasson-Miranda-Zung normal form can be used. We then set $\Sigma^n$ a (small enough) $n$-dimensional submanifold of $M$ which intersects transversally the foliation $\cF$ at $A_0$, with $V_{FF-X^{n-2}}(\cU(\Lambda_0)) \cap \Sigma^n = \emptyset$. We set $\Omega := \{ \Lambda_c \in \cF \ |\  \Lambda_c \cap \Sigma^n \neq \emptyset \}$. The open set $\Omega$ is in $\cU(\Lambda_0)$.

We have that $G^{-1} \circ F$ is a global moment map for the foliation $\cF$ on $\Omega$. On $\cU_{MZ}$, $G^{-1} \circ F = (q_1,q_2,\xi_1,\ldots,\xi_{n-2})$, so $G^{-1} \circ F$ is an extention of $Q_{FF-X^{n-2}}$ to $\Omega$. We can now simply forget $G^{-1}$ and just take $F$ for a global momentum map that extends $Q_{FF-X^{n-2}}$, and restrict the system to $\Omega \simeq W$, with $\Gamma \simeq W_0$.

\ouf

For all critical leaves in $L_{FF-X^{n-2}}(\Omega)$, all the $f_1$-orbits are homoclinic orbits. Since $q_2,\xi_1,\ldots,\xi_{n-2}$ have $2\pi$-periodic flows on $\cU_{MZ}$ and since for their extension $f_2,\ldots,f_n$, all the $f_i$'s Poisson-commute, $\check{F}$ yield a $\TT^{n-1}$-action on $\Omega$ that commutes with the flow of $f_1$. The $\TT^{n-2}$-action induced by $\check{F}^2$ is free everywhere on $\Omega$ as told by item.1 of Theorem~\ref{theo:A-L_with_sing}, while the $S^1$-action of $f_2$ is free on $P_{X^n}(\Omega)$.

Since the leaves of $\cF$ are compact, for $A \in \Sigma^n \setminus L_{FF-X^{n-2}}(\Omega)$ we can define $A'$ the point of first intersection of the $f_1$-orbit of $A$ with the orbit by $\check{F}$, and $\tau_1$ the time of intersection. Let $(\tau_2,\ldots,\tau_n)$ be the multi-time needed to come back from $A'$ to $A$ with the flows of $f_2,f_3,\ldots,f_n$, hence closing the trajectory

\[ \phi_{f_2}^{\tau_2} \circ \cdots \circ \phi_{f_n}^{\tau_n} (A') = A. \]

Since the joint flow of $F$ is transitive these times depend only of the Lagrangian torus and not of $A$, and thus, only of the values $\Bv$ of $F$. For any regular value $\Bv$, the set of all the $(\alpha_1,\ldots,\alpha_n) \in \RR^n$ such that $\alpha_1 \chi_{f_1}(\Bv) + \ldots + \alpha_n \chi_{f_n} (\Bv)$ has a $1$-periodic flow is a sublattice of $\RR^n$ called the period lattice. The following matrix

\[ \cR = 
\begin{pmatrix}
\tau \\
r_2 \\          
\vdots \\
\vdots \\
r_n 
\end{pmatrix} =  \begin{pmatrix}
    \tau_1 (v) & \cdots &\cdots & \cdots & \tau_n (v) \\
    0 & 2\pi & 0 & \cdots & 0 \\
    \vdots & 0 & \ddots & \ddots & \vdots \\
    \vdots & \vdots &\ddots & \ddots & 0 \\
    0 & 0 & \cdots & 0 & 2\pi
    \end{pmatrix}
\]
forms a $\ZZ$-basis of the period lattice ($\tau_1 > 0$ by definition). These vectors can also be seen as a basis of cycles of the Lagrangian tori foliation $\cF$ on $\Omega$. The next proposition proves the second item of Theorem~\ref{theo:AA_coord_around_FF_value}: it gives the singular behavior of the basis as $\Bv$ tends to $\Gamma$.

\begin{proposition} \label{prop:regular_1-form}
 Let us fix a determination of the complex logarithm: $\ln(w)$, where $w = v_1 + i v_2$. Then the following quantities 

\begin{itemize}
 \item $\sigma_1 (\Bv) = \tau_1(\Bv) + \Re e (\ln (w)) \in \RR$,
 \item $\sigma_2 (\Bv) = \tau_2(\Bv) - \Im m (\ln (w)) \in {\RR}/{2\pi\ZZ} $,
 \item $\sigma_3 (\Bv) = \tau_3(\Bv) \in \RR / 2\pi\ZZ \ ,\ldots,\ \sigma_n (\Bv) = \tau_n(\Bv) \in {\RR}/{2\pi\ZZ}$,
\end{itemize}

are defined on $W \setminus W_0$ and can be extended to smooth and single-valued functions of $\Bv$ in $W$. Moreover, the differential form 

\[ \sigma = \sum_{i=1}^n \sigma_i dv_i \] is closed in $W$.

\end{proposition}

\begin{proof}
 
We fix some small $\varepsilon >0$ and we set 
\[ \Sigma_u^\Balpha := \{ z_1 = \varepsilon, z_2 \text{ small }, \Btheta = (\theta_1,\ldots,\theta_{\kx}) = \Balpha , \Bxi \in [-\varepsilon,\varepsilon]^{\kx} \} \subseteq \cU_{MZ} \]
\[ \Sigma_s^\Balpha := \{ z_1 \text{ small }, z_2 = \varepsilon, \Btheta = (\theta_1,\ldots,\theta_{\kx}) = \Balpha , \Bxi \in [-\varepsilon,\varepsilon]^{\kx} \} \subseteq \cU_{MZ} .\]

These are stable and unstable local submanifolds for the dynamic of the $f_1$-flot near a critical point $m^{\Balpha}_{\Bxi} = (z_1 = 0,z_2 = 0,\Btheta = (\theta_1,\ldots,\theta_{n-2}) = \Balpha , \Bxi) \in CrP_{FF-X^{n-2}}(\Omega)$. They are $n$-dimensional submanifolds intersecting transversally the foliation $\cF$ on $\Omega$. Thus, the intersections $A(\Bv,\Balpha) := \Lambda_\Bv \cap \Sigma_u^\Balpha$ and $B(\Bv,\Balpha) := \Lambda_\Bv \cap \Sigma_s^\Balpha$ are points of $M$ in the same $\chi_{f_1}$-orbit for $\Bv \in V_{X^n}(\Omega)$, at least for $A(\Bv,\Balpha)$ and $B(\Bv,\Balpha)$ in $L_{X^n}(\Omega)$. They are well-defined smooth functions of $\Bv$ and $\Balpha$.

The $\TT^{n-1}$-orbits of $A(\Bv,\Balpha)$ and $B(\Bv,\Balpha)$ are transversal to the Hamiltonian flow of $f_1$, thus one can define $\tau^{A,B}_1(\Bv,\Balpha)$ as the time necessary for the Hamiltonian flow of $f_1$ starting at $A(\Bv,\Balpha)$ (which flows outside of $\cU_{MZ}$), to make first hit to the $\TT^{n-1}$-orbit of $B(\Bv,\Balpha)$. Let us call this first hit $B'=(b_1',b_2',\Balpha + \Btheta,\Bxi)$. Since the $\TT^{n-1}$-orbit of $B$ is in $\cU_{MZ}$, we know that in it $f_2 = q_2$, $f_3 = \xi_3$, $\ldots$,$f_n = \xi_n$, so we have the explicit expression for the time needed to get back to $B$, which we call $\tau^{A,B}_2(\Bv,\alpha)$, $\ldots$, $\tau^{A,B}_n(\Bv,\alpha)$:
\[ \tau^{A,B}_2(\Bv,\Balpha) = \arg(b_1') \text{ and for } j=3 \ldots n \ ,\ \tau^{A,B}_j(\Bv,\Balpha) = 2\pi - \theta_j' \]

Since the $f_i$'s commute, the $\tau^{A,B}_j$ are smooth, single-valued functions of $\Bv$ only. We can now interchange the roles of $A$ and $B$, and thus, of $\Sigma_u^\Balpha$ and $\Sigma_s^\Balpha$, to define the times $\tau^{B,A}_j (v)$ for $j=1,\ldots,n$. The joint flow of $F$ now takes place inside $\cU_{MZ}$ where $L_{FF-X^{n-2}}(\Omega) \cap \cU_{MZ}$ is a codimension-$2$ manifold, so for $\Bv \in  \Gamma$, the quantities $\tau^{B,A}_1 (v),\tau^{B,A}_2 (v),\ldots,\tau^{B,A}_n (v)$ cannot be defined a priori.

However, one can see that in the definition, $\tau^{B,A}_1 (v)$ and $\tau^{B,A}_2 (v)$ do not depend  of the value of $(f_3,\ldots,f_n)$: in Eliasson-Miranda-Zung theorem, the local model is a \emph{direct product} of the Eliasson normal form for the focus-focus and the \act coordinate for the transversal component. Moreover, since everywhere it is defined, for $j = 3,\ldots,n, \tau^{B,A}_j (v) = 0$, its limit when $\Bv \to \Gamma$ must be $0$ also.

With the explicit formulaes of the Hamiltonian flow of $q_1$ and $q_2$ given in Table~\ref{tab:properties_of_crit_points}, we know that $\tau^{B,A}_1 (v)$ and  $\tau^{B,A}_2 (v)$ satisfy the following equation:

\begin{equation} \label{eq:tau}
(e^{\tau^{B,A}_1 + i \tau^{B,A}_2 } b_1, e^{-\tau^{B,A}_1 + i \tau^{B,A}_2 } b_2,\Btheta,\Bxi) = (a_1,a_2,0,\ldots,0,0,\ldots,0) . 
\end{equation}

We also have the equations: $a_1 = \varepsilon \ ,\ b_2 = \varepsilon \text{ and } \bar{a_1} a_2 = \bar{b_1} b_2 = w $. Here we introduce our determination of the complex logarithm to give the solution of~\ref{eq:tau}: 

\[ \tau^{B,A}_1 + i \tau^{B,A}_2  = \ln(\frac{a_1}{b_1}) = \ln(\varepsilon . \frac{\varepsilon}{\bar{w}}) = \ln(\varepsilon^2) - \ln(\bar{w}) .\]

Writing now $ \tau_1 + i \tau_2 =  (\tau^{A,B}_1 + \tau^{B,A}_1) + i (\tau^{A,B}_2 + \tau^{B,A}_2)$, we can refer to the statement announced on Proposition~\ref{prop:regular_1-form} concerning $\sigma_1$ and $\sigma_2$:

\[
\begin{aligned} 
 \sigma_1 + i \sigma_2 & = \tau_1(\Bv) + \Re e (\ln (w)) + i(\tau_2(\Bv) - \Im m (\ln (w))) \\
  \ & = \tau^{A,B}_1 + i \tau^{A,B}_2 + \tau^{B,A}_1 + i \tau^{B,A}_2 + \ln(\bar{w})  \\
  \ & = \tau^{A,B}_1 + i \tau^{A,B}_2 + \ln(\varepsilon^2) - \ln(\bar{w}) + \ln(\bar{w})  \\
  \ & = \tau^{A,B}_1 + i \tau^{A,B}_2 + \ln(\varepsilon^2) .
\end{aligned}
\]

This last quantity is smooth with respect to $\Bv$. Since for $j = 3,\ldots,n$, $\sigma_j (\Bv) = \tau_j (\Bv) = \tau^{A,B}_j (\Bv)$ is also smooth, this shows the first statement of Proposition~\ref{prop:regular_1-form}.

\ouf

Let us now show that for regular values, the 1-form $\tau = \sum_{i=1}^n \tau_i dv_i $ is closed. For this we fix a regular value $\Bc \in V_{X^n}(\Omega)$ and introduce the following action integral

\[ \cA (\Bv) = \int_{\gamma_\Bv} \alpha \] 
with $\alpha$ a Liouville 1-form defined on a tubular \neigh $\cV(\Lambda_\Bc) \subseteq \Omega$ of $\Lambda_\Bc$ ($\omega = d\alpha$), while, for $\Bv \in U:= F(\cV(\Bc))$,  $\Bv \mapsto \gamma_\Bv \subseteq \Lambda_\Bv$ is a smooth family of loops with the same homotopy class in $\Lambda_\Bv$ as the joint flow of $F$ at the times $(\tau_1 (\Bv), \tau_2 (\Bv), \ldots, \tau_n(\Bv))$. The integral $\cA$ only depends of $\Bv$ as $\gamma_\Bv \subseteq \Lambda_\Bv$, which is Lagrangian (this is another statement of Mineur formula).

A consequence of Darboux-Weinstein that one can found in~\cite{Weinstein-SymplecticmanifoldsandLagrangian-1971}, is the following general lemma

\begin{lemma} \label{lemma:identify_Lag_1-form}
 Each Lagrangian submanifold in a tubular \neigh $\cV(\Lambda_\Bc)$ can be canonically associated with a closed 1-form on $\Lambda_\Bc$.
\end{lemma}


\begin{proof}
 The exponential map provides a diffeomorphism between $\cV(\Lambda_\Bc)$ and the normal bundle 
\[ N \Lambda_\Bc = \bigsqcup_{p \in \Lambda_\Bc} \! \frac{T_p M}{T_p \Lambda_\Bc}. \] 
The latter can be identified with $T^* \Lambda_\Bc$ using the symplectic form: for $m \in \Lambda_\Bc$ and $X_m \in T_m M$, we define 
\[ \tilde{\omega}[X_m]_m := \left. (\imath_{X_m} \omega_m) \right|_{T \Lambda_\Bc} \in T^*_m \Lambda_\Bc .\]

Since $T_m M = T_m \Lambda_\Bc \oplus N_m \Lambda_\Bc$ with $\Lambda_\Bc$ Lagrangian, the map 
\[ \begin{aligned} TM & \to T^* \Lambda_\Bc \\ X & \mapsto \tilde{\omega}[X] \end{aligned} \]
is linear, and $\tilde{\omega}[X]$ is non-zero if and only if the projection of $X$ on $N \Lambda_\Bc$ is non-zero as a vector field. 

Now, an infinitesimal deformation of the submanifold $\Lambda_\Bc$ is a vector field of $\cV(\Lambda_\Bc)$ transversal to $\Lambda_\Bc$, that is, a section $X \in N \Lambda_\Bc$. This infinitesimal deformation is to be performed in the space of Lagrangian submanifolds and only if 

\[\begin{aligned}
    \ddt \left. \left. \left[ {\phi_X^t}^{ *}  \omega \right] \right|_{t=0} \right|_{T\Lambda_\Bc \times T\Lambda_\Bc} & = 0 \\
    \left[ \Lie_X \omega \right]_{T\Lambda_\Bc \times T\Lambda_\Bc} & = 0 \\
\left.  \left[ d\imath_X \omega + \imath_X \underbrace{d\omega}_{= 0} \right]
\right|_{T\Lambda_\Bc \times T\Lambda_\Bc} & = 0 \\
    d \left[ (\imath_X \omega)|_{T\Lambda_\Bc} \right] & = d \tilde{\omega}[X] = 0,
  \end{aligned}
\]
that is, if and ony if the associated 1-form $\tilde{\omega} [X]$ is closed, if and only if $X$ is locally Hamiltonian in our case.

\end{proof}

In our case the foliation $\cF = (\Lambda_\Bv)_{\Bv \in U}$ is given by the fibers of the moment map $F|_U$. Its deformation map is $\Bv \mapsto \Lambda_\Bv$ and the associated 1-form to the infinitesimal deformation verifies

\[ \left. \tilde{\omega} \left[ \frac{\partial }{\partial v_i} \Lambda_\Bv \right] \right|_{\Bv=\Bc} = \kappa_i(\Bc) \] 
where $\kappa_i$ is the closed 1-form on $\cV(\Lambda_\Bc)$ defined by: $\imath_{\chi_{f_j}} \kappa_i = \delta_{i,j}$. In other words, the integral of $\kappa_i$ along a trajectory of the flow of $F$ measures the increasing of time $t_j$ along this trajectory. We now show the following formula linking the variation 1-form to the infinitesimal variation of the action $\cA$ 

\begin{equation} \label{equ:deformation=action}
 \frac{\partial \cA}{\partial v_i} (\Bc) = \int_{\gamma_\Bc} \kappa_i
\end{equation}

by proving that 
\[ \left. \frac{\partial }{\partial v_i} (\gamma_\Bv^* \alpha) \right|_{\Bv=\Bc} \text{ and } \gamma_\Bc^* \kappa_i \; \text{are cohomologous on} \; S^1. \]

We have

\[ \frac{\partial}{\partial v_i} (\gamma_\Bv^* \alpha) = \gamma_\Bv^* \left[ \Lie_{\frac{\partial \gamma_\Bv}{\partial v_i}} \alpha \right] 
= \gamma_\Bv^* \left[ \imath_{\frac{\partial \gamma_\Bv}{\partial v_i} } d\alpha + d \imath_{\frac{\partial \gamma_\Bv}{\partial v_i}} \alpha \right]\]

To each $\gamma_\Bv$ corresponds a unique Lagrangian submanifold $\Lambda_\Bv$, hence for all $p \in  \gamma_\Bc$, $\frac{\partial \gamma_\Bv}{\partial v_i} |_{\Bv=\Bc} (p) = \frac{\partial }{\partial v_i} \Lambda_\Bv |_{\Bv=\Bc} (p)$. The vector splits into two components $X^t_\Bc (p)$ and $X^n_\Bc (p)$, with $X^t_\Bc (p) \in T_p \Lambda_\Bc$ and $X^n_\Bc (p) \in N_p \Lambda_\Bc$. The normal vector $X^n_\Bc $ is by definition the infinitesimal deformation of $\cF$ at $\Bc$ along the direction $\frac{\partial }{\partial v_i}$, that is, $\kappa_i$. We then have

\[ \left. \left( \gamma_\Bv^* \left[ \imath_{\frac{\partial \gamma_\Bv}{\partial v_i} } d\alpha \right] \right) \right|_{\Bv = \Bc}
= \gamma_\Bc^* \left( \underbrace{\imath_{X^t_\Bc} d\alpha}_{= 0 \text{ on } T_p \Lambda_\Bc } + \imath_{X^n_\Bc} \underbrace{d\alpha}_{ = \omega} \right) = \gamma_\Bc^* \tilde{\omega}[\frac{\partial}{\partial v_i} \Lambda_\Bv]|_{\Bv = \Bc}  = \gamma_\Bc^* \kappa_i. \]

Thus $ \left. \frac{\partial }{\partial v_i} (\gamma_\Bv^* \alpha) \right|_{\Bv=c} - \gamma_\Bc^* \kappa_i = d \gamma_\Bc^*\alpha(\frac{\partial \gamma_\Bv}{\partial v_i}) $ is exact: the two 1-forms are cohomologous.
\end{proof}

 Since $\gamma_\Bc$ has the same homotopy class as the joint flow of $F$ at the times $(\tau_1 (\Bc), \tau_2 (\Bc),\ldots,\tau_n (\Bc))$, we have

\[ d\cA (\Bc)= \sum_{i=1}^n \frac{\partial \cA}{\partial v_i}(\Bc) dv_i = \sum_{i=1}^n \tau_i(\Bc) dv_i = \tau(\Bc) \]

Thus, we can now forget about $\Bc$: $\tau$ is a closed 1-form of the variable $\Bv$, $\cA$ is the action integral, defined for all $\Bv \in F(L_{X^n}(\cU(\Lambda_0)))$. The function $\ln (w)$ is also closed as a holomorphic 1-form of the variable $w$ and thus, $\sigma = \tau + \ln (w)$ is closed for all regular values, and by continuation, for all $\Bv \in F(\cU(\Lambda_0))$. Taking the primitive of $ln(w)$, we define the

\begin{definition}
 Let $S$ be the unique smooth function of $\Bv$ defined on $F(\cU(\Lambda_0))$ such that $dS = \sigma$ and $S(0) = 0$. The Taylor serie of $S$ in $(v_1,v_2)$ at $(0,0,v_3,\ldots,v_n)$ can be written as 
 \[ (S)^\infty = \sum_{j_1 = 0}^\infty \sum_{j_2 = 0}^\infty \frac{\partial S}{\partial v_1^{j_1} \partial v_2^{j_2} } (0,0,v_3,\ldots,v_n) v_1^{j_1} v_2^{j_2}\]
 
 In accordance with~\cite{SVN-semiglobalinvariants-2003} we call this double sum the symplectic invariant of the nodal locus $\Gamma$ and we have:
 
 \[ \cA (\Bv) = S(\Bv) - \Re e (w \ln w - w) \]
\end{definition}

This concludes the proof of Theorem~\ref{theo:AA_coord_around_FF_value}.

\end{proof}

\subsection{Computation of the monodromy}

In order to exhibit the monodromy phenomenon, we first mention that in~\cite{Wacheux-LocalmodelsofS-TIntSys-2014}, we show that $V_{X^n} (\Omega)$ is not simply connected. We take a circle 

\[ C := \{ C(t) = (v_1 \cos(t),v_2 \sin(t),v_3,\ldots,v_n) | t \in [0,2\pi] \} \subseteq F(M)\]

with $\Bv$ small enough so that it is in in $F(\Omega)$. Each point of $C$ is a regular value, so to each $t$ corresponds a $n$-torus. We can take any basis of cycle for each torus, that is, a $\ZZ$-basis for the period lattice. For $C(0)$, we take the basis $\cR$ introduced in the proof of Theorem~\ref{theo:AA_coord_around_FF_value}, and we follow it by homotopy as $t$ goes from $0$ to $2\pi$. From Proposition~\ref{prop:regular_1-form}, it turns out that as $t$ varies, $\tau$ is deformed as $\tau + (0,t,0,\ldots,0)$, the other vectors staying unchanged. Thus, at $t=2\pi$, we are back to $C(0)$, but our basis of covectors is now 

\[ \cR' = \begin{pmatrix} 
\tau + r_2 \\
r_2 \\
\vdots \\
r_n
\end{pmatrix}
\]

Hence, we can write the following corollary of Theorem~\ref{theo:AA_coord_around_FF_value}

\begin{corollary}
 With the same hypothesis as Theorem~\ref{theo:AA_coord_around_FF_value}, the topological monodromy matrix is:
 
 \[ \begin{pmatrix}
    1 & 1 & 0 & \ldots & 0 \\
    0 & 1 & 0 & \ldots & 0 \\
    0 & 0 & \ulcorner & \; & \urcorner \\  
    \vdots & \vdots & \; &  I_{n-2} & \; \\
    0  & 0 & \llcorner & \; & \lrcorner    
    
    \end{pmatrix}
 \]
 
\end{corollary}

\subsection{The general case} \label{subsection:the_general_case}

Now for the general case , we consider a critical point $m$ of Williamson type $FF-X^{\kx}-E^{\ke}$ with $\ke \neq 0$. We can always assume the last $\ke$ components to be the elliptic ones. Using again Theorem~\ref{theo:A-L_with_sing}, we have that for $q_1 = q_2 = 0$ and $e_i \neq 0, i = 1..\ke$, the singularity is a $FF-X^{n-2}$ singularity. 

On the other hand, although we only have proved it for $2n=6$, we conjecture that

\[ \left( \cU_\Bbbk(\Lambda_0) := \cup_{\kx' \leq \kx } L_{\Bbbk'} (\cU(\Lambda_0)) , 
\omega_\Bbbk := \omega|_{T^*\cU_\Bbbk(\Lambda_0)}, F_\Bbbk = F|_{\cU_\Bbbk} \right) \]
is a semi-toric IHS. It is the stratification of the semi-toric system by the Williamson type, which shall be the subject of another article. If we admit the conjecture, we have that in this system, $m$ is a $FF-X^{\kx}$ critical point. We can hence apply the result above, and for $S_\Bbbk(v_1,v_2,v_3,\ldots,v_{\kx+2}) = S_\Bbbk(\Bv_\Bbbk)$ the regularized action coordinate associated to it, we have

\[ S_\Bbbk(\Bv_\Bbbk) = \lim_{\Bv \to (\Bv_\Bbbk,0)} S(\Bv). \]



Note that, while we can still describe the foliation for an elliptic singularity, we do not have \act coordinates: polar coordinates are not defined at the origin. Concerning the monodromy, we have for the restricted IHS the following matrix 

\[ 
\begin{pmatrix}
    1 & 1 & 0 & \ldots & 0 \\
    0 & 1 & 0 & \ldots & 0 \\
    0 & 0 & \ulcorner & \; & \urcorner \\  
    \vdots & \vdots & \; &  I_{\kx} & \; \\
    0  & 0 & \llcorner & \; & \lrcorner    
    
    \end{pmatrix}
 \]

\section{Conclusion and perspectives}

In this article, we have established the general formula for the Taylor invariant introduced by \SVN in~\cite{SVN-semiglobalinvariants-2003} and the monodromy matrix. Some of the results needed to prove that it is indeed an invariant are proved given in Taylor invariant we have provided several results using local and leaf-wise model of singular integrable systems. These techniques are helpful in less friendly settings. In almost-toric systems of higher complexity, for instance, we can apply the same tools. However, it is not clear what could be a general formulation of this result, to what extent we can simply mimetize the proof we did here.


 There are several global results in our research program (see \cite{PelayoSVN-ConstructingIntSysOfSemitoricType-2011} and \cite{PelayoRatiuSVN-SymplecticBifTheoryForIntegrableSystems-2011} for a description ) that necessitate more conceptual tools than the local models like we used here. In that sense, the aim of articles~\cite{Wacheux-about_image_s-t_mm-2014} and~\cite{Zung_Wacheux-Trop_strad:intrinsinc_convexity_of_A-T_Sys-2014} are to establish these result in a fit framework. 

\section*{Bibliography}



\begin{thebibliography}{PRVN11}

\bibitem[Ati82]{Atiyah-ConvexityandCommuting-1982}
M.~F. Atiyah.
\newblock Convexity and commuting hamiltonians.
\newblock {\em Bulletin of the London Mathematical Society}, 14(1):1--15, 1982.

\bibitem[BF04]{BolsinovFomenko-book}
A.~V. Bolsinov and A.~T. Fomenko.
\newblock {\em Integrable {H}amiltonian systems; Geometry, topology,
  classification}.
\newblock Chapman \& Hall, 2004.
\newblock Translated from the 1999 Russian original.

\bibitem[Cha86]{Chaperon-GeoDiff-SingSysDyn-Asterisque-1986}
M.~Chaperon.
\newblock {\em G{\'e}om{\'e}trie diff{\'e}rentielle et singularit{\'e}s de
  syst{\`e}mes dynamiques}.
\newblock Number 138-139. 1986.

\bibitem[Cha12]{Chaperon-SmoothFFaSimpleProof-2012}
M.~Chaperon.
\newblock Normalisation of the smooth focus-focus: a simple proof. with an
  appendix by jiang kai.
\newblock {\em Acta Mathematica Vietnamica}, page~8, 2012.

\bibitem[DD87]{DazordDelzant-Leproblemegeneral-1987}
P.~Dazord and T.~Delzant.
\newblock {Le probl\`eme g\'en\'eral des variables actions-angles.}
\newblock {\em J. Differ. Geom.}, 26:223--251, 1987.

\bibitem[Del88]{Delzant-Hamiltoniensperiodiqueset-1988}
T.~Delzant.
\newblock {Hamiltoniens p\'eriodiques et images convexes de l'application
  moment.}
\newblock {\em Bull. Soc. Math. Fr.}, 116(3):315--339, 1988.

\bibitem[Del90]{Delzant-ClassificationActions-1990}
T.~Delzant.
\newblock Classification des actions hamiltoniennes compl{\'e}tement
  int{\'e}grables de rang deux.
\newblock {\em Annals of Global Analysis and Geometry}, 8:87--112, 1990.
\newblock 10.1007/BF00055020.

\bibitem[DH82]{DuistermaatHeckman-variationincohomology-1982}
J.~J. Duistermaat and G.~J. Heckman.
\newblock On the variation in the cohomology of the symplectic form of the
  reduced phase space.
\newblock {\em Inventiones Mathematicae}, 69:259--268, 1982.
\newblock 10.1007/BF01399506.

\bibitem[DM91]{DufourMolino-CompactActionRnVarAA-1991}
J.-P. Dufour and P.~Molino.
\newblock Compactification d'actions de {${\bf R}^n$} et variables action-angle
  avec singularit\'es.
\newblock In {\em Symplectic geometry, groupoids, and integrable systems
  ({B}erkeley, {CA}, 1989)}, volume~20 of {\em Math. Sci. Res. Inst. Publ.},
  pages 151--167. Springer, New York, 1991.

\bibitem[DMT94]{DufourMolinoToulet-ClassIHSetInvFomenko-1994}
Jean-Paul Dufour, Pierre Molino, and Anne Toulet.
\newblock Classification des syst\`emes int\'egrables en dimension {$2$} et
  invariants des mod\`eles de {F}omenko.
\newblock {\em C. R. Acad. Sci. Paris S\'er. I Math.}, 318(10):949--952, 1994.

\bibitem[Dui80]{Duistermaat-globalactionangle-1980}
J.~J. Duistermaat.
\newblock On global action-angle coordinates.
\newblock {\em Communications on Pure and Applied Mathematics}, 33(6):687--706,
  1980.

\bibitem[dVV79]{ColinVey-LemmeMorseIsochore-1979}
Y.~Colin de~Verdiere and J.~Vey.
\newblock Le lemme de morse isochore.
\newblock {\em Topology}, 18(4):283 -- 293, 1979.

\bibitem[Eli84]{Eliasson-Thesis-1984}
L.~H. Eliasson.
\newblock {\em Hamiltonian systems with Poisson commuting integrals}.
\newblock PhD thesis, Stockholm, 1984.

\bibitem[Eli90]{Eliasson-NormalformsHamiltonian-1990}
L.~H. Eliasson.
\newblock {Normal forms for Hamiltonian systems with Poisson commuting
  integrals -- elliptic case}.
\newblock {\em Commentarii Mathematici Helvetici}, 65:4--35, 1990.

\bibitem[GS82]{GuilleminSternberg-ConvexitypropertiesI-1982}
V.~Guillemin and S.~Sternberg.
\newblock Convexity properties of the moment mapping.
\newblock {\em Inventiones Mathematicae}, 67:491--513, 1982.
\newblock 10.1007/BF01398933.

\bibitem[GS84]{GuilleminSternberg-ConvexitypropertiesII-1984}
V.~Guillemin and S.~Sternberg.
\newblock Convexity properties of the moment mapping. ii.
\newblock {\em Inventiones Mathematicae}, 77:533--546, 1984.
\newblock 10.1007/BF01388837.

\bibitem[KT01]{KarshonTolman-Centeredcomplexity1HamTorusAction-2001}
Y.~Karshon and S.~Tolman.
\newblock Centered complexity one {H}amiltonian torus actions.
\newblock 353(12):4831--4861 (electronic), 2001.

\bibitem[KT03]{KarshonTolman-CompleteinvariantsforHamT_actions_tall-2003}
Y.~Karshon and S.~Tolman.
\newblock Complete invariants for {H}amiltonian torus actions with two
  dimensional quotients.
\newblock {\em J. Symplectic Geom.}, 2(1):25--82, 2003.

\bibitem[KT11]{KarshonTolman-ClassificationofHamiltonian-2011}
Y.~Karshon and S.~Tolman.
\newblock {Classification of Hamiltonian torus actions with two dimensional
  quotients}.
\newblock {\em ArXiv e-prints}, September 2011.

\bibitem[LS10]{LeungSymington-Almosttoricsymplectic-2010}
N.~C. Leung and M.~Symington.
\newblock Almost toric symplectic four-manifolds.
\newblock {\em J. Symplectic Geom.}, 8(2):143--187, 2010.

\bibitem[MZ04]{MirandaZung-Equivariantnormalform-2004}
E.~Miranda and Nguyen~Tien Zung.
\newblock Equivariant normal form for nondegenerate singular orbits of
  integrable hamiltonian systems.
\newblock {\em Annales Scientifiques de l'{\'E}cole Normale Sup{\'e}rieure},
  37(6):819 -- 839, 2004.

\bibitem[PRVN11]{PelayoRatiuSVN-SymplecticBifTheoryForIntegrableSystems-2011}
A.~Pelayo, T.~S. Ratiu, and S.~V. V{\~u}~Ng\d{o}c.
\newblock {Symplectic bifurcation theory for integrable systems}.
\newblock {\em ArXiv e-prints}, August 2011.

\bibitem[PVN09]{PelayoSVN-Semitoricintegrablesystems-2009}
A.~Pelayo and S.~V{\~u}~Ng\d{o}c.
\newblock Semitoric integrable systems on symplectic 4-manifolds.
\newblock {\em Inventiones Mathematicae}, 177:571--597, 2009.
\newblock 10.1007/s00222-009-0190-x.

\bibitem[PVN11]{PelayoSVN-ConstructingIntSysOfSemitoricType-2011}
A.~Pelayo and S.~V{\~u}~Ng\d{o}c.
\newblock Constructing integrable systems of semitoric type.
\newblock {\em Acta Mathematica}, 206:93--125, 2011.

\bibitem[Sym01]{Symington-4from2-2001}
M.~Symington.
\newblock Four dimensions from two in symplectic topology.
\newblock {\em Proc. Sympos. Pure Math., 71, Amer. Math. Soc., Providence,
  RI.}, Topology and geometry of manifolds (Athens, GA, 2001)(71):153--208,
  2001.

\bibitem[Vey78]{Vey-SurCertainsSystemes-1978}
J.~Vey.
\newblock Sur certains systemes dynamiques separables.
\newblock {\em American Journal of Mathematics}, 100(3):pp. 591--614, 1978.

\bibitem[VN00]{SVN-BohrSommerfeld-2000}
S.~V{\~u}~Ng\d{o}c.
\newblock Bohr-{S}ommerfeld conditions for integrable systems with critical
  manifolds of focus-focus type.
\newblock {\em Comm. Pure Appl. Math.}, 53(2):143--217, 2000.

\bibitem[VN03]{SVN-semiglobalinvariants-2003}
S.~V{\~u}~Ng\d{o}c.
\newblock On semi-global invariants for focus-focus singularities.
\newblock {\em Topology}, 42(2):365--380, 2003.

\bibitem[VN07]{SVN-Momentpolytopessymplectic-2007}
S.~V{\~u}~Ng\d{o}c.
\newblock Moment polytopes for symplectic manifolds with monodromy.
\newblock {\em Advances in Mathematics}, 208(2):909 -- 934, 2007.

\bibitem[VNW13]{SVNWacheux-SmoothNF_for_IHS_near_ff_sing-2013}
S.~V{\~u}~Ng\d{o}c and C.~Wacheux.
\newblock Smooth normal forms for integrable hamiltonian systems near a
  focus-focus singularity.
\newblock {\em Acta Mathematica Vietnamica}, 38(1):107--122, 2013.

\bibitem[Wac14a]{Wacheux-about_image_s-t_mm-2014}
Wacheux.
\newblock About the image of semi-toric moment maps.
\newblock 2014.

\bibitem[Wac14b]{Wacheux-LocalmodelsofS-TIntSys-2014}
Wacheux.
\newblock Local model of semi-toric integrable systems.
\newblock 2014.
\newblock http://arxiv.org/abs/1408.1166.

\bibitem[Wei71]{Weinstein-SymplecticmanifoldsandLagrangian-1971}
A.~Weinstein.
\newblock Symplectic manifolds and their lagrangian submanifolds.
\newblock {\em Advances in Mathematics}, 6:329--346, 1971.

\bibitem[Wil36]{Williamson-OnAlgPbLinNF-1936}
J.~Williamson.
\newblock On the algebraic problem concerning the normal form of linear
  dynamical systems.
\newblock {\em American Journal of Mathematics}, 58(1):141--163, 1936.

\bibitem[Zun96]{Zung-SymplecticTopology_I-1996}
Nguyen~Tien Zung.
\newblock Symplectic topology of integrable hamiltonian systems. i :
  Arnold-liouville with singularities.
\newblock {\em Compos. Math.}, 101(2):179--215, 1996.

\bibitem[Zun97]{Zung-Anoteonfocusfocus-1997}
Nguyen~Tien Zung.
\newblock A note on focus-focus singularities.
\newblock {\em Differential Geometry and its Applications}, 7(2):123 -- 130,
  1997.

\bibitem[Zun02]{Zung-Anothernotefocus-2002}
Nguyen~Tien Zung.
\newblock Another note on focus-focus singularities.
\newblock {\em Letter in Mathematical Physics}, 60(1):87--99, 2002.

\bibitem[ZW14]{Zung_Wacheux-Trop_strad:intrinsinc_convexity_of_A-T_Sys-2014}
N.~T. Zung and Wacheux.
\newblock Tropical stradispace : intrinsic convexity of almost-toric systems.
\newblock 2014.

\end{thebibliography}

\end{document}